\documentclass[a4paper]{article}

\usepackage[english]{babel}

\usepackage[utf8]{inputenc}
\usepackage[utf8]{inputenc}
\usepackage{amsmath}
\usepackage{graphicx}
\usepackage{amssymb}
\usepackage{amsthm}
\usepackage{tikz-cd}
\usepackage{mathrsfs}
\usepackage[colorinlistoftodos]{todonotes}
\usepackage{enumitem}
\usepackage{yfonts}
\usepackage{ dsfont }
\usepackage{MnSymbol}
\usepackage{slashed}

\title{Uniform Skoda integrability and  Calabi-Yau degeneration}

\author{Yang Li}

\date{\today}
\newtheorem{thm}{Theorem}[section]
\newtheorem{lem}[thm]{Lemma}

\theoremstyle{definition}

\newtheorem{cor}[thm]{Corollary}

\newtheorem{rmk}[thm]{Remark}
\newtheorem{prop}[thm]{Proposition}

\newtheorem*{Acknowledgement}{Acknowledgement}

\newcommand{\cf}{\emph{cf.} }

\newcommand{\C}{\mathbb{C}}

\newcommand{\norm}[1]{\left\lVert#1\right\rVert}

\def\Xint#1{\mathchoice
	{\XXint\displaystyle\textstyle{#1}}%
	{\XXint\textstyle\scriptstyle{#1}}%
	{\XXint\scriptstyle\scriptscriptstyle{#1}}%
	{\XXint\scriptscriptstyle\scriptscriptstyle{#1}}%
	\!\int}
\def\XXint#1#2#3{{\setbox0=\hbox{$#1{#2#3}{\int}$ }
		\vcenter{\hbox{$#2#3$ }}\kern-.6\wd0}}

\def\dashint{\Xint-}

\begin{document}
	\maketitle

\begin{abstract}
We study polarised algebraic degenerations of Calabi-Yau manifolds. We prove a uniform Skoda type estimate, and a uniform $L^\infty$-estimate for the Calabi-Yau K\"ahler potentials.
\end{abstract}

\section{Introduction}

Let $(Y,\omega)$ be a compact K\"ahler manifold, and $d\mu$ be a measure on $Y$. We say $(Y,\omega, d\mu)$ satisfies the \emph{Skoda type inequality}, if for any K\"ahler potential $u \in PSH(Y,\omega)$ normalised to $\sup u=0$, 
\begin{equation}\label{Skodaassumption}
\int_Y e^{-\alpha u} d\mu \leq A, 
\end{equation}
where $\alpha, A$ are independent of  $u$. A prototype theorem is

\begin{thm}\label{Skodastandardversion}\cite{Tian}
	On a fixed compact K\"ahler $(Y, \omega)$, the Skoda type inequality holds for $d\mu=\omega^n$.
\end{thm}

\begin{rmk}
Here the supremum of all such $\alpha$ 
 is known as Tian's alpha invariant, important for existence questions of K\"ahler-Einstein metrics. 	
\end{rmk}

We are interested  in keeping track of these constants $\alpha, A$ as  $(Y,\omega,d\mu)$ varies. The main theme of this paper is that oftentimes the Skoda constants can be chosen uniformly for quite flexible choices of probability measures $d\mu$, even when the complex structure degenerates severely. 	In the literature $\alpha$ is much studied (\cf \cite{Tian}\cite{GZ}), and a very recent preprint \cite{Eleonora} made aware to the author after the completion of this work contains a  uniform estimate for both $\alpha, A$ in the related context of K\"ahler-Einstein manifolds.


Our main application is to algebraic degenerations of Calabi-Yau manifolds. 
We work over $\C$. Let $S$ be a smooth affine algebraic curve, with a point $0\in S$.
An \emph{algebraic degeneration family} is given by a submersive projective morphism $\pi: X\to S\setminus \{0\}$ with smooth connected $n$-dimensional fibres $X_t$ for $t\in S\setminus \{0\}$. A \emph{polarisation} is given by an ample line bundle $L$ over $X$; the sections of a sufficiently high power of
$L$ induces an embedding $X\to \mathbb{CP}^N$, hence a Fubini-Study metric $\omega_X$ on $(X, L)$. For $0<|t|\ll 1$, 
a fixed choice of $\omega_X$ induces \emph{rescaled} background metrics $\omega_t= \frac{1}{ |\log |t| |  } \omega_X|_{ X_t }$ on $X_t$ in the class
$\frac{1}{|\log |t| |   }c_1(L)  $.

 A \emph{model} of $X$ is a normal flat projective $S$-scheme $\mathcal{X}$ which agrees with $\pi: X\to S\setminus \{0\}$ over the punctured curve. It is called a \emph{semistable snc model} if $\mathcal{X}$ is smooth, 
 the central fibre over $0\in S$ is \emph{reduced} and is a simple normal crossing divisor in $\mathcal{X}$. By the semistable reduction theorem \cite[chapter 2]{ToroidalembeddingsI}, after finite base change to another smooth algebraic curve $S'$, we can always find some semistable snc model for the degeneration family $X\times_S (S'\setminus \{0\})$. Everything here is quasi-projective.

 We say the degeneration family is \emph{Calabi-Yau} if there is a trivialising section $\Omega$ of the canonical bundle $K_X$. Over a small disc $\mathbb{D}_t$ around $0\in S$, this induces holomorphic volume forms $\Omega_t$ on $X_t$ via $\Omega= dt\wedge \Omega_t$. The \emph{normalised Calabi-Yau measure} on $X_t$ is the probability measure
 \begin{equation}\label{CalabiYaumeasureeqn}
 d\mu_t= \frac{   \Omega_t \wedge \overline{\Omega}_t }{   \int_{X_t} \Omega_t \wedge \overline{\Omega}_t   }.
 \end{equation}

Our main result is

\begin{thm}(Uniform Skoda estimate)\label{UniformSkodamaintheorem}
	Given a polarised algebraic Calabi-Yau degeneration family $\pi: X\to S\setminus \{0\}$ as above. Then there are uniform positive constants $\alpha, A$ independent of $t$ for $0<|t|\ll 1$, such that for the normalised Calabi-Yau measures $d\mu_t$,
	\[
	\int_{X_t} e^{-\alpha u}  d\mu_t \leq A, \quad \forall u\in PSH(X_t,\omega_t) \text{ with } \sup_{X_t} u=0.
	\]

\end{thm}

This is proved by reducing to the semistable snc model case, and prove a general Skoda type estimate there (\cf Theorem \ref{UniformSkodasemistable}). A major consequence, readily reaped using Kolodziej's estimate (\cf Theorem \ref{pluripotentialthm1}), is

\begin{thm}\label{UniformLinfty}
(Uniform $L^\infty$-estimate) Let $\phi_t$ be the K\"ahler potential of the Calabi-Yau metric in the class $(X_t, [\omega_t])$, namely
\[
\frac{(\omega_t+ \sqrt{-1} \partial\bar{\partial}\phi)^n }{ \int_{X_t} \omega_t^n  }= d\mu_t,\quad \sup_{X_t} \phi_t=0.
\]
Then $\norm{ \phi_t}_{L^\infty} \leq C$ independent of $t$ for $0<|t|\ll 1$.

\end{thm}


\begin{rmk}
Applications of pluripotential theory to Calabi-Yau metrics when the K\"ahler class is degenerating can be found in \cite{EGZ2}, which is used further in \cite{Tosatti}. Our main results generalize certain aspects of \cite{LiSYZ} which focuses on  degenerating projective hypersurfaces near the large complex structure limit.
\end{rmk}

\begin{Acknowledgement}
The author is a 2020 Clay Research Fellow, currently based at the Institute for Advanced Study. He thanks Song Sun and Simon Donaldson for discussions, and Sebastien Boucksom, Eleonora Di Nezza, and Valentino Tossati for comments.
\end{Acknowledgement}

\section{Uniform Skoda inequality}\label{UniformSkodachapter}

We work in the context of semistable simple normal crossing (snc) models. Concretely,
let $\pi: \mathcal{X}\to \mathbb{D}_t$ be a flat projective family of $n$-dimensional varieties over a small disc $\mathbb{D}_t$, such that the total space $\mathcal{X}$ is smooth, $\pi$ is a submersion over the punctured disc with connected fibres, the central fibre $X_0$ is \emph{reduced} and is an snc divisor in $\mathcal{X}$. Denote the components of $X_0$ as $E_i$ with $i\in I$. We equip $\mathcal{X}$ with a fixed background K\"ahler metric $\omega_{\mathcal{X}}$, inducing a distance function $d_{\omega_{\mathcal{X}}}$. This induces a family of \emph{rescaled} K\"ahler metrics $\omega_t= \frac{1}{ |\log |t||  } \omega_{ \mathcal{X}  }|_{X_t  }$. We shall derive a uniform Skoda type estimate (\ref{Skodaassumption}) for $(X_t, \omega_t, d\mu_t)$, where $d\mu_t$ belongs to a natural class of measures. The main result is Theorem \ref{UniformSkodasemistable}.

\subsection{Quantitative stratification and good test functions}

There is a \emph{quantitative stratification} on any smooth fibre $X_t$ induced by the intersection pattern of $E_i$: for $J\subset I$ such that $E_J=\cap_{i\in J} E_i\neq \emptyset$, the corresponding statum is 
\[
E_J^0= \{   x\in X_t | d_{\omega_{\mathcal{X}}}( x, E_J )\lesssim \epsilon     \} \setminus \{    x\in X_t | d_{\omega_{\mathcal{X}}}( x, E_{J'} )\lesssim \epsilon ,  \text{ some } J'\supsetneq J     \},
\] 
namely a small `$\epsilon$-tubular neighbourhood' of $E_J$ minus the deeper strata. For $J=\{ i\}$ we write $E_i^0= E_{\{i\}  }^0$. Here the disc $\mathbb{D}_t$ and the small parameter $\epsilon\ll 1$ can be shrinked for convenience; the essential thing is that all parameters should be independent of the coordinate $t$.

It is useful to introduce local coordinates $\{ z_i \}_0^{n}$ around $E_J\subset \mathcal{X}$, such that $z_0, \ldots, z_p$ with $p=|J|-1$ are the local defining equations of $E_j$ for $j\in J$, and locally the  fibration map is $t=z_0\ldots z_p$.  Then up to uniform equivalence, locally
\begin{equation*}
\omega_{ \mathcal{X} }\sim \sum_0^{n} \sqrt{-1}dz_i\wedge d\bar{z}_i.
\end{equation*}

The rest of this section is devoted to the construction of good test functions.
Given any of these divisors $E_0$, we can find a nonnegative function $h=h_{E_0}$ on $\mathcal{X}$, such that
\begin{itemize}
\item 

In the local charts near $E_0$ with $z_0$ being the defining function for $E_0$,
\[
h= |z_0|^2 \tilde{h}(z_0, \ldots z_n)
\]
for some positive smooth function $\tilde{h}$;
\item Away from $E_0$ the function $h$ is comparable to 1.
\end{itemize}
 
We observe 
\begin{itemize}
	\item
The form $\partial \bar{\partial} \log h= \partial \bar{\partial} \log \tilde{h}$ extends smoothly;
\item
For $|t|^2\ll h\lesssim \delta\ll 1$ inside $X_t$, so that $|z_0|\gg |t|$, by a local calculation  near $E_J$ with $0\in J$, 
\[
\begin{split}
 &\sqrt{-1}   \partial  \log h   \wedge \bar{\partial } \log h   \wedge \omega_{ \mathcal{X}  }|_{X_t}^{n-1} \geq \frac{ \sqrt{-1} }{2|z_0|^2} dz_0\wedge d\bar{z}_0\wedge\omega_{ \mathcal{X}  }|_{X_t}^{n-1}  
 \\
 &\gtrsim   \min\{  \frac{1}{|z_0|^2}, \max_{1\leq i\leq p} |z_i|^{-2}  \} \omega_{ \mathcal{X}  }|_{X_t}^n
 \\
 &
 \gtrsim  
  \min\{  \frac{1}{h}, h^{1/p} |t|^{-2/p}  \}\omega_{ \mathcal{X}  }|_{X_t}^n
  \\
  & 
 \gtrsim  
 \min\{  \frac{1}{h}, h^{1/n} |t|^{-2/n}  \}\omega_{ \mathcal{X}  }|_{X_t}^n .
\end{split}
\]
Here in the first line we need to fix $\delta\ll 1$ so that the effect of $\partial \log h$ is dominated by $d\log z_0$.  The second line uses that for $1\leq k\leq p$, the volume forms on $X_t$
\[
\frac{1}{|z_0|^2} dz_0\wedge d\bar{z}_0\wedge
\prod_{j\neq k, 1\leq j\leq n} \sqrt{-1}dz_j\wedge d\bar{z}_j
\sim
\frac{1}{|z_k|^2} \prod_{ 1\leq j\leq n} \sqrt{-1}dz_j\wedge d\bar{z}_j,
\]
and the third line uses $|t|=|z_0\ldots z_p|\sim h^{1/2}|z_1\ldots z_p|$.

\item
On $X_t$ the function $h\gtrsim |t|^2$. 
The region $\{ |t|^2\sim h  \}\subset X_t$ can be identified as $E_0^0$, namely the vicinity of $E_0$ away from deeper strata. Here
\[
\begin{split}
&\sqrt{-1}   \partial  \log h   \wedge \bar{\partial } \log h   \wedge \omega_{ \mathcal{X}  }|_{X_t}^{n-1} \geq 0,
\\
& 
\sqrt{-1}   \partial  \bar{\partial } \log h   \wedge \omega_{ \mathcal{X}  }|_{X_t}^{n-1} \gtrsim -\omega_{ \mathcal{X}  }|_{X_t}^{n}  .
\end{split}
\]

\end{itemize}

\begin{lem}\label{goodtestfunction}
	(Good test function) Given the divisor $E_0$, we can choose a $C^2$ test function $v$ on $X_t$ such that the following hold uniformly for small $t\neq 0$:
	\begin{itemize}
		\item $v$ is zero for $h\geq \delta$.
		
		\item Globally $0\leq v\leq -\log |t|$.

		\item
		For any divisor $E_j$ intersecting $E_0$, there is a subset of $E_j^0$ with measure at least $C_2$ on which $\sqrt{-1} \partial \bar{\partial}v \wedge \omega_{ \mathcal{X}  }|_{X_t}^{n-1}  \geq C_3\omega_{ \mathcal{X}  }|_{X_t}^n $.

		\item 
		For $C_4 |t|^2 \leq h\leq \delta$, the form $\sqrt{-1}\partial \bar{\partial} v\wedge \omega_{ \mathcal{X}  }|_{X_t}^{n-1} \geq 0$.

		\item 
		For $h\leq C_4 |t|^2$, the form $\sqrt{-1}\partial \bar{\partial} v \wedge \omega_{ \mathcal{X}  }|_{X_t}^{n-1}\geq -C_5\omega_{ \mathcal{X}  }|_{X_t}^n $.
	
	\end{itemize}
\end{lem}

\begin{proof}
We seek the test function in the form $v=\Phi\circ \log h$ for some convex, non-increasing, non-negative $C^2$-function $\Phi$. Compute
\[
\partial \bar{\partial} v=  \Phi'' \partial  \log h   \wedge \bar{\partial } \log h   + \Phi' ( \partial \bar{\partial } \log \tilde{h}),
\]
so using the properties of $h$ above,
\[
\sqrt{-1} \partial \bar{\partial}v \wedge \omega_{ \mathcal{X}  }|_{X_t}^{n-1} 
\geq
\begin{cases}
\left( \Phi'' C_1' \min\{  \frac{1}{h}, h^{1/n} |t|^{-2/n}  \} +\Phi'C_2'  \right)    \omega_{ \mathcal{X}  }|_{X_t}^n     , \quad &|t|^2\lesssim h\leq \delta,
\\
C_3' \Phi' \omega_{ \mathcal{X}  }|_{X_t}^n, \quad & h\lesssim |t|^2.
\end{cases}
\]

To satisfy our conditions on $v$, it is enough to have
\begin{itemize}
\item
$\Phi(x)=0$ for $x\geq \log \delta$.
\item
$|\Phi'(x)|\lesssim 1$ for $2\log |t|\lesssim x\leq \log \delta$. 
\item
$ -\frac{d}{dx} \log |\Phi'|=  \frac{\Phi''}{|\Phi'|  }  \geq C_4' \max\{ h, h^{-1/n} |t|^{2/n}  \}$ for $ h=e^x\leq \delta,$ where
$C_4'> C_2'/C_1'$. Morever, for $x<\delta$, we need $\Phi'<0$ so that  
$\sqrt{-1} \partial \bar{\partial}v \wedge \omega_{ \mathcal{X}  }|_{X_t}^{n-1}$ has some strict positivity for $\delta/2<h<\delta$. Notice convexity of $\Phi$ is a consequence of these conditions.

\end{itemize}

To construct such $\Phi$, we can prescribe the behaviour near $x=\log \delta$ by
$
\Phi'(x)= -e^{1/(x-\log \delta  )}
$ for $x< \log \delta$,
and match this with a solution to
\[
-\frac{d}{dx} \log |\Phi'|  = C_4' \max\{e^x, e^{-x/n} |t|^{2/n}  \}, \quad x<\log \delta
\]
for some large enough $C_4'$, such that $\Phi'$ remains $C^1$ at the matching point. Integration shows that $ |\Phi'|$ remains uniformly bounded at $h\sim |t|^2$, or equivalently $x\sim 2\log |t|$.
\end{proof}

\subsection{Convexity}

Consider $u\in PSH(X_t, \omega_t)$ normalised to $\sup_{X_t} u=0$. Equivalently, we can cover $X_t$ by a bounded number of charts as before, and use the
 local potentials of $\omega_{\mathcal{X} }$ to represent $u$  as a collection of local plurisubharmonic (psh) functions $\{ u_\beta\}$ with $|u_\beta-u|\leq C$.

\begin{lem}(Convexity)\label{pshconvexity}
	Let $\phi$ be any psh function on the open subset of $  \{   1<|z_i|<\Lambda, i=1,\ldots p, |z_k|<1, k=p+1, \ldots n \}\subset (\C^*)^p\times \C^{n-p}$. Then the function \[
	\bar{\phi}(x_1, \ldots x_n)= \frac{1}{(2\pi)^n} \int_{D(1)^{n-p} } \prod_{p+1}^n \sqrt{-1} dz_k \wedge d\bar{z}_k 
		\int_{T^p} \phi( e^{x_1+i\theta_1}, \ldots e^{x_p+i\theta_p}) d\theta_1\ldots d\theta_p
	\]
	is convex. 
\end{lem}

\begin{proof}
For any choice of $\theta_i$  the function  $\phi(z_1 e^{i\theta_1}, \ldots z_p e^{i\theta_n}, z_{p+1},\ldots, z_n)$ is psh, since the $T^p$-action on $(\C^*)^p$ is holomorphic. Thus the average function $\bar{\phi}$ is also psh as a function of $z_1,\ldots z_p$. Any $T^p$-invariant psh function must be convex in the log coordinates, because for $x_i=\log |z_i|$,
	\[
	\sqrt{-1}\partial \bar{\partial} \bar{\phi}= \frac{1}{4} \sum \frac{\partial^2 \bar{\phi}}{  \partial x_i \partial x_j  } \sqrt{-1} d\log z_i \wedge d\overline{\log z_j} \geq 0.
	\]
\end{proof}

\subsection{Harnack type inequality}

\begin{lem}
	(Almost maximum on top strata)
	For $u\in PSH(X_t, \omega_t)$ normalised to $\sup_{X_t} u=0$, there is some $i\in I$, such that
	\[
	\sup_{E_i^0} u \geq -C, \quad \int_{E_i^0} u \omega_{ \mathcal{X} }|_{X_t}^n \geq -C'.
	\]

\end{lem}

\begin{proof}
	Let the global maximum of $u$ be achieved at $q_0\in E_J^0$, and denote the local potential of $u$ as $u_\beta$. Without loss of generality $u_\beta\leq 0$. We have $u_\beta(q_0) \geq -C$ since $|u-u_\beta|\leq C$. Applying the mean value inequality around $q_0$, we find that the local average function $\bar{u}_\beta$ produced in Lemma \ref{pshconvexity} satisfies $\sup \bar{u}_\beta\geq -C$ for another uniform constant $C$. By the convexity of $\bar{u}_\beta$ its sup is almost achieved at the boundary of the chart, which is contained in a union of less deep strata $E_{J'}^0$ with $J'\subsetneq J$. Thus we can find a point $q'$ with $u(q')\geq -C$ that belongs to a less deep stratum; an induction shows that there is some $i\in I$, such that $\sup _{E_i^0} u \geq -C$.

For the $L^1$-bound we recall the following Harnack inequality argument. Suppose a coordinate ball $B(q, 3R)$ is contained in a local chart in a small neighbourhood of $E_i^0$.
Applying the mean value inequality to the local psh function associated to $u$, we see for $y\in B(q,R)$ that
\[
u(y) \leq  C+ \dashint_{B(y,2R)} u \lesssim 1+ \dashint_{B(q,R)} u.
\]
hence the Harnack inequality
\[
\dashint_{B(q,R)} |u |\lesssim 1+ \inf_{B(q, R)} (-u). 
\]
Applying this to a chain of balls connecting any two points in $E_i^0$ gives the $L^1$-bound $\int_{E_i^0} u \omega_{ \mathcal{X} }|_{X_t}^n \geq -C'$; the bound is uniform because the number of balls involved in the chain can be controlled independent of $t$.
\end{proof}

\begin{prop}(Almost maximum on top strata II)\label{AlmostmaximumII}
	There is a uniform lower bound for all $|t|\ll 1$ and all $i\in I$:
\begin{equation}
\sup_{E_i^0} u \geq -C, \quad \int_{E_i^0} u \omega_{ \mathcal{X} }|_{X_t}^n \geq -C'.
\end{equation}
\end{prop}

\begin{proof}
The $L^1$-estimate follows from the sup estimate as above, so the real problem is to transfer bounds between different $E_i^0$. This is nontrivial because the necks connecting $E_i^0$ with each other are highly degenerate.

Given one divisor $E_0$ such that $
 \int_{E_0^0} u \omega_{ \mathcal{X} }|_{X_t}^n \geq -C,
$
we produce a good test function $v$ by Lemma \ref{goodtestfunction}. Integrating by parts,
\[
\int_{X_t} v \sqrt{-1}\partial \bar{\partial} u \wedge \omega_{\mathcal{X} }|_{X_t}^{n-1} = \int_{X_t} u \sqrt{-1}\partial \bar{\partial} v \wedge \omega_{\mathcal{X} }|_{X_t}^{n-1}.
\]
The LHS is the difference of $\int_{X_t} v(\omega_t+ \sqrt{-1}\partial \bar{\partial} u) \wedge \omega_{\mathcal{X} }|_{X_t}^{n-1}$ and $\int_{X_t} v\omega_t \wedge \omega_{\mathcal{X} }|_{X_t}^{n-1}$, and since $-\log |t|\gtrsim v\geq 0$ both terms are bounded between $0$ and $C$. Thus
\[
|\int_{X_t} u \sqrt{-1}\partial \bar{\partial} v \wedge \omega_{\mathcal{X} }|_{X_t}^{n-1}|\leq C.
\]

Now the form
$\sqrt{-1}\partial \bar{\partial} v \wedge \omega_{ \mathcal{X}  }|_{X_t}^{n-1}$ can only be negative on $\{ h\sim |t|^2  \}=E_0^0$, and is bounded below by $-C  \omega_{ \mathcal{X}  }|_{X_t}^{n}$. Thus the positive part of the signed measure $u \sqrt{-1}\partial \bar{\partial} v \wedge \omega_{\mathcal{X} }|_{X_t}^{n-1}$ has total mass controlled by $ \int_{E_0^0} |u| \omega_{ \mathcal{X} }|_{X_t}^n \leq C$.
Consequently, the negative part of the signed measure must also have total mass $\leq C$.

By construction, for any divisor $E_j$ intersecting $E_0$ there is a nontrivial amount of $\sqrt{-1}\partial \bar{\partial} v \wedge \omega_{\mathcal{X} }|_{X_t}^{n-1}$-measure inside $E_j^0$. This forces $\sup_{E_j^0} u \geq -C$. To summarize, we have transferred the sup bound from $E_0^0$ to any $E_j^0$ with $E_j\cap E_0\neq \emptyset$. Since the central fibre $X_0$ is connected, in at most $|I|$ steps 
this sup bound is transferred to all $E_i^0$ with $i\in I$.
\end{proof}

\begin{rmk}
This proof is inspired by the intersection theoretic argument of \cite[section 6.1]{Boucksomsemipositive}, which can be viewed as a non-archimedean analogue.
\end{rmk}

\subsection{Local $L^1$ estimate}

Given a local chart on $E_J^0$ with $\C^*$-coordinates $z_1,\ldots z_p$ and $\C$-coordinates $z_{p+1}, \ldots, z_n$, and a point $q$ therein, we shall refer to the subregion
\[
	\{  \frac{1}{2} |z_i(q)|  \lesssim|z_i| \lesssim 2|z_i(q)|  , \quad 1\leq i\leq p    \}
\]
as a \emph{log scale}.

\begin{lem}\label{LocalL1estimate}
	(Local $L^1$-estimate)
Within every log scale there is a uniform bound on the $L^1$-average integral
\[
\dashint_{loc} |u| \prod_1^p \sqrt{-1}d\log z_i \wedge d\log \bar{z}_i \wedge \prod_{p+1}^n \sqrt{-1}d z_k \wedge d \bar{z}_k \leq C.
\]
\end{lem}

\begin{proof}
We induct on the depth of the strata. For $p=0$ this follows from Prop. \ref{AlmostmaximumII}. So let us assume the bound is achieved for depth $<p$. For a given chart, we consider the local psh function $u_\beta$ associated to $u$ and produce the convex average function $\bar{u}_\beta$ as in Lemma \ref{pshconvexity}. Since a definite neighbourhood of the boundary of the chart lies inside less deep strata, we know that near the boundary $|\bar{u}_\beta|\leq C$ by the induction hypothesis and the convexity condition. Using convexity again in the interior of the chart we see $|\bar{u}_\beta|\leq C$ in the whole chart.

Within any log scale, by construction the local average
$
\dashint_{loc} (u_\beta- \bar{u}_\beta)=0.
$
But $u_\beta\leq C$ by $u\leq 0$, hence
\[\dashint_{loc} |u_\beta-\bar{u}_\beta| \lesssim \dashint_{loc} (u_\beta- \bar{u}_\beta )_+ \leq C.
\]
Using $|u-u_\beta|\leq C$ we conclude the local $L^1$-estimate on $u$.
\end{proof}

\subsection{Local Skoda  estimate}\label{Skodainequalitybackgroundsection}

We recall a basic version of the \emph{Skoda inequality}:

\begin{prop}\label{Skodabasicversion}
	(\cf \cite[Thm 3.1]{Zeriahi}) If $\phi$ is psh on $B_2\subset \C^n$, with $\int_{B_2} |\phi| \omega_E^n \leq 1$ with respect to the standard Euclidean metric $\omega_E$, then
	there are dimensional constants $\alpha$, $C$, such that 
	\[
 \int_{B_1} e^{-\alpha \phi} \omega_E^n \leq  C.
	\]

\end{prop}

Applying this with Lemma \ref{LocalL1estimate},

\begin{cor}
(Local Skoda estimate) Within every log scale, there are uniform positive constants $\alpha$ and $C$, such that
\[
\dashint_{loc} e^{-\alpha u} \prod_1^p \sqrt{-1}d\log z_i \wedge d\log \bar{z}_i \wedge \prod_{p+1}^n \sqrt{-1}d z_k \wedge d \bar{z}_k \leq C.
\]
\end{cor}

\subsection{Uniform global Skoda estimate}

We are interested in the following class of measures, motivated by Calabi-Yau measures (\cf section \ref{CalabiYaumeasure}). Let $a_i$ be non-negative real numbers assigned to $i\in I$, with $\min a_i=0$. Let
\[
m= \max\{ |J|-1:  E_J\neq \emptyset, a_i=0 \text{ for } i\in J     \}.
\]
We say the measures $d\mu_t$ on $X_t$ satisfy a \emph{uniform upper bound of class $(a_i)$}, if on the local charts of each $E_J^0$,
\begin{equation}\label{measureupperbound}
d\mu_t \leq \frac{ C}{ |\log |t||^m  } |z_0|^{2a_0}\cdots |z_p|^{2a_p} \prod_1^p  \sqrt{-1}d\log z_i \wedge d\log \bar{z}_i \wedge \prod_{p+1}^n \sqrt{-1}d z_k \wedge d \bar{z}_k.
\end{equation}
 The normalisation factor ensures $\int_{X_t}d\mu_t \leq C$ independent of $t$, by a straightforward local calculation.

\begin{thm}\label{UniformSkodasemistable}
(Uniform Skoda estimate) Suppose the measures $d\mu_t$ on $X_t$ satisfy a uniform upper bound of class $(a_i)$. Then there are uniform positive constants $\alpha$ and $A$, such that
\[
\int_{X_t} e^{-\alpha u}  d\mu_t \leq A, \quad \forall u\in PSH(X_t,\omega_t) \text{ with } \sup_{X_t} u=0.
\]

\end{thm}

\begin{proof}
We choose the charts so that each point on $X_t$ is covered by $\leq C$ log scales. Summing over the local Skoda estimates from all log scales, $\int_{X_t} e^{-\alpha u}d\mu_t$ is bounded by
\[
\begin{split}
 &\frac{C}{ |\log |t||^m  }  \sum_{\text{log scales}} \int_{loc} |z_0|^{2a_0}\ldots |z_p|^{2a_p}
\prod_1^p \sqrt{-1}d\log z_i \wedge d\log \bar{z}_i \wedge \prod_{p+1}^n \sqrt{-1}d z_k \wedge d \bar{z}_k 
\\
&\leq C.
\end{split}
\]
\end{proof}

\section{Application to Calabi-Yau degeneration}

We work in the setting of polarised algebraic degeneration of Calabi-Yau manifolds, as in the Introduction.

\subsection{Calabi-Yau measure}\label{CalabiYaumeasure}

The Calabi-Yau measure (\ref{CalabiYaumeasureeqn}) is studied thoroughly in \cite{Boucksom1}, but it is illustrative to recall it explicitly on a semistable snc model $\mathcal{X}$. The discussion is local on the base, and we will follow the notations of section \ref{UniformSkodachapter}, e.g. the components of the central fibre are denoted as $E_i$ for $i\in I$.

The canonical divisor $K_{ \mathcal{X} }=\sum_i a_i E_i$ is supported on the central fibre, since $K_X$ is trivialised. Multiplying $\Omega$ by a power of $t$, which does not change $d\mu_t$, we may assume $\min a_i=0$.
In the local coordinates around $E_J$ away from the deeper strata, 
\[
\Omega= f_J \prod_0^p z_i^{a_i} dz_i \wedge \prod_{p+1}^n dz_j
\]
for some nowhere vanishing local holomorphic function $f_J$. Since $t= z_0\ldots z_p$,
\[
\Omega_t=  f_J z_0^{a_0}\ldots z_p^{a_p} \prod_1^p  d\log z_i \wedge \prod_{p+1}^n dz_j,
\]
hence
\[
\sqrt{-1}^{n^2}\Omega_t \wedge \overline{\Omega}_t= |f_J|^2 |z_0|^{2a_0}\ldots |z_p|^{2a_p} \prod_1^p \sqrt{-1}d\log z_i \wedge d\overline{\log z_i} \wedge \prod_{p+1} ^n\sqrt{-1}d z_j \wedge d\overline{z}_j. 
\]
The total measure $\int_{X_t} \sqrt{-1}^{n^2}\Omega_t \wedge \overline{\Omega}_t$ is of the order $O(|\log |t| |^m  )$ where
\[
m= \max\{ |J|-1:  E_J\neq \emptyset, a_i=0 \text{ for } i\in J     \}.
\]
Thus $d\mu_t $ satisfies a uniform upper bound of class $(a_i)$ (\cf (\ref{measureupperbound})).

\subsection{Uniform Skoda estimate}

We now prove the main theorem \ref{UniformSkodamaintheorem}.

\begin{proof}
First we observe that the choice of the Fubini-Study metric $\omega_X$ is immaterial. Given any two choices, the relative K\"ahler potential between them is bounded by $O(|\log |t| |)$ for $0<|t|\ll 1$, because the pole order of a section near $t=0$ must be finite. Thus the relative K\"ahler potential between two choices of $\omega_t$ is bounded by $O(1)$ independent of $t$, which affects the Skoda constant $A$ but not its uniform nature.

We now pass to a finite base change and find a semistable reduction. The Calabi-Yau measure $d\mu_t$ on $X_t$ is independent of the parametrisation of the base, and is preserved under finite base change. Thus it is enough to prove it assuming $\omega_X$ agrees with a smooth K\"ahler metric on a semistable snc model $\mathcal{X}$; this is a special case of Theorem \ref{UniformSkodasemistable}.
\end{proof}

\subsection{Uniform  $L^\infty$-estimate}\label{Toolsfrompsh}

We recall the following result proved using Kolodziej's pluripotential theoretic methods (\cf \cite[section 2.2]{LiSYZ} for an exposition based on \cite{EGZ}\cite{EGZ2}):

\begin{thm}\label{pluripotentialthm1}
	Let $(Y, \omega)$ be a compact K\"ahler manifold, and the K\"ahler potential $\phi$ solves the complex Monge-Amp\`ere equation
	\[
\frac{	(\omega+\sqrt{-1} \partial \bar{\partial} \phi )^n}{  \int_Y \omega^n }=d\mu, \quad \sup \phi=0.
	\]
 Assume there are positive constants $\alpha, A$, such that the Skoda type estimate (\ref{Skodaassumption}) holds
 for $(Y,\omega, d\mu)$:
	\begin{equation*}
	\int_Y e^{-\alpha u}  d\mu \leq A, \quad \forall u\in PSH(Y,\omega) \text{ with } \sup_Y u=0.
	\end{equation*}
Then  $\norm{\phi}_{C^0} \leq C(n,\alpha,A)$.

\end{thm}

The uniform $L^\infty$-estimate for the Calabi-Yau potentials in Theorem \ref{UniformLinfty} is an immediate consequence.

\end{document}